\documentclass[a4paper,12pt]{amsart}
\usepackage[utf8]{inputenc}
\usepackage{amssymb,amsthm,amsmath,enumerate}
\usepackage[margin=3cm]{geometry}

\begin{document}
\subjclass[2010]{53B40}
\keywords{Finsler manifold, metrizability, affine transformation, isometry, holonomy}
\author{D.~Cs.~Kert\'esz}
\title{Affinely rigid Finsler manifolds}
\begin{abstract}
A Finsler function $F$ is \emph{affinely rigid} if its canonical spray is uniquely metrizable, in the sense that if $\bar F$ is another Finsler function whose canonical spray is $S$, then $d(F/\bar F)=0$. In this short note we explore some sufficient conditions for a Finsler function to be affinely rigid, and discuss open problems.
\end{abstract}
\maketitle

\newtheorem{thm}{Theorem}
\newtheorem{lmm}[thm]{Lemma}
\newtheorem{crl}[thm]{Corollary}
\newtheorem{prp}[thm]{Proposition}

\theoremstyle{definition}
\newtheorem{rmrk}[thm]{Remark}

\newcommand{\vv}{\mathbf{v}}
\newcommand{\lsz}{\upharpoonright}
\newcommand{\NN}{\mathbb{N}}
\newcommand{\RR}{\mathbb{R}}
\newcommand{\vl}{^\mathsf{v}}
\newcommand{\tl}{^\mathsf{c}}
\newcommand{\slt}{\mathring{T}}
\newcommand{\vm}{\mathfrak{X}}
\newcommand{\hl}{^\mathsf{h}}
\newcommand{\ir}[1]{\mathcal{#1}}
\newcommand{\hll}[1]{\widetilde{#1}}
\newcommand{\prcf}[3]{
\frac{\partial #1}{\partial #2^{#3}}
}
\newcommand{\ic}[2]{
^{#1}_{#2}
}
\newcommand{\inv}{^{-1}}
\newcommand{\fels}[5]{#1#2#3{#4}\dots#4#1#2{#5}}

\section{Motivation}
The problem of affine rigidity rises naturally when one attempts to examine the relation between affine\footnote{preserves geodesics as parametrized curves} and isometric\footnote{preserves the Finsler function} transformations of Finsler manifolds; in particular, in the problem of characterizing Finsler manifolds that admit no proper affine transformations. As to the Riemannian case, it is well-known, that the isometry and affinity groups of a complete irreducible Riemannian manifold coincide, except for the one-dimensional Euclidean space \cite{MR0076395}. Most of the steps of the proof can be translated to a Finsler manifold, with one crucial exception: \emph{the irreducibility implies that the holonomy group uniquely determines the inner products on the tangent spaces up to a constant factor}. This proposition does not have a clear analogue in Finsler geometry, as there is no notion of irreducibility. So the problem rises: find characterizations of Finsler manifolds, whose `holonomy structure' determines the Finsler function, up to a constant factor. Since the `holonomy structure' is described by the canonical spray determined by the Finsler function, we arrive to the question in the abstract.

We may construct the holonomy group of a Finsler manifold $(M,F)$, analogously to that of a Riemannian manifold, using the parallel translation with respect to the Berwald connection. In this way, for a fixed point $p\in M$, we obtain a subgroup $\operatorname{Hol}_p$ of the group of smooth diffeomorphisms of $T_pM\setminus\{0\}$. Each element of $\operatorname{Hol}_p$ is a $1^+$-homogeneous diffeomorphism of $T_pM\setminus\{0\}$ , and it preserves the Finsler function. For a connected Finsler manifold, $\operatorname{Hol}_p$ and $\operatorname{Hol}_q$ are of course isomorphic for any $p$ and $q$ in $M$, so we may speak of \emph{the} holonomy group of a (connected) Finsler manifold. These groups can be vastly different from the holonomy groups of Riemannian manifolds, as they can be infinite-dimensional (see, e.g., \cite{MR2917275}). Non-Berwald Landsberg manifolds have non-Riemannian, but finite dimensional holonomy groups \cite{MR2066448}, however, it is still not known whether such Finsler manifolds exist.

The following observation is immediate:
\begin{prp}\label{prp:trnzero}
 Let $(M,F)$ be a Finsler manifold. If $\operatorname{Hol}_p$ acts transitively on the unit sphere $U(T_pM):=\{v\in T_pM\mid F(v)=1\}$, then $(M,F)$ is affinely rigid.
\end{prp}

Irreducible Riemannian manifolds are affinely rigid. By Berger's holonomy theorem \cite{MR0079806,MR2150392} there are irreducible Riemannian manifolds whose holonomy groups are not transitive on the unit sphere, which implies that the reverse of Proposition~\ref{prp:trnzero} is not true, and the transitivity of $\operatorname{Hol}_p$ should be replaced by a weaker condition.

It is worth noting that Riemannian manifolds are `much more rigid' than Finsler manifolds. In the Riemannian case, the norms on the tangent spaces are required to be quadratic functions, thus they are uniquely determined by their Hessian at any point. Finsler functions allow much more freedom, therefore characterizing affinely rigid Finsler manifolds is expected to be more difficult than the Riemannian ones.

\section{Preliminaries}
Let $(M,F)$ be a Finsler manifold. We denote by $\mathring\tau\colon\slt M\to M$ the slit tangent bundle, and $V\slt M:=\ker\mathring\tau_*$ is the vertical subbundle of $T\slt M$. There is a unique subbundle $H\slt M$ of $T\slt M$ satisfying the following properties:
\begin{enumerate}
 \item $V\slt M\oplus H\slt M = T\slt M$.
 \item For a vector field $X\in\vm(M)$, denote by $X\hl$ the unique smooth section of $H\slt M$ which is $\mathring\tau$ related to $X$. Then $X\hl$ is $1^+$-homogeneous, i.e., $[C,X\hl]=0$ where $C$ is the Liouville vector field;
 \item For all $X,Y\in\vm(M)$, $[X\hl,Y\vl]-[Y\hl,X\vl]=[X,Y]\vl$ (the torsion vanishes).
 \item $H\slt M\subset\ker dF$.
\end{enumerate}
For details, we refer to \cite{CSF}.

\section{Subspace fields}

We recall a few concepts and results about subspace fields from \cite{Michor}. In the cited reference, they are called singular distributions. Given a manifold $M$, we denote by $\vm_{loc}(M)$ the set of vector fields that are defined only on an open subset of $M$. Suppose that on a manifold $M$, for each $p\in M$ we have a subspace $E_p$ of $T_pM$. Then the disjoint union $E=\bigsqcup_{p\in M} E_p$ is a \emph{subspace field} on $M$. We denote by $\vm_E$ the (smooth by assumption) local vector fields in $\vm_{loc}(M)$, that take values only in $E$. We say that a subset $\ir V$ of $\vm_E$ \emph{spans} $E$, if at each $p\in M$, $E_p$ is the linear span of $\{X(p)\in T_pM\mid X\in \ir V\}$. Here we agree that the linear span of the empty set is the zero element of the vector space. We say that $E$ is \emph{smooth}, if it is spanned by $\vm_E$.

An \emph{integral manifold} of a smooth subspace field $E$ is an immersed submanifold $N$ with immersion $i\colon N\to M$, such that $i_*(T_pN)=E_{i(p)}$ for all $p\in N$. It turns out that these integral manifolds are actually initial submanifolds, so we need not to specify the immersion $i$.

A subset $\ir V$ of $\vm_{loc}(M)$ is \emph{stable}, if for any $X,Y\in\ir V$, the local vector field $(\operatorname{FL}\ic Xt)_\#(Y)$\footnote{here $\operatorname{Fl}^X$ denotes the flow of $X$, and `\#' stands for push-forward, e.g., $\varphi_\# X = \varphi_*\circ X\circ\varphi\inv$} is also in $\ir V$.
For a set $\ir W\subset \vm_{loc}(M)$, $\ir S(\ir W)$ denotes the set of local vector fields of the form
\[
 (\operatorname{Fl}\ic{X_1}{t_1}\circ \dots \circ \operatorname{Fl}\ic{X_k}{t_k})_\# Y,
\]
where $\fels X_1,k,Y\in \ir W$, $k\in \NN$. Then $\ir S(\ir W)$ is the smallest stable subset of $\vm_{loc}(M)$ that contains $\ir W$. According to  \cite[3.24~Lemma]{Michor}, the smooth subspace field $E$ spanned by $\ir S(\ir W)$ is \emph{integrable}, in the sense that any point of $M$ is contained in an integral manifold of $E$.

The following observation is from \cite{MR2886203}.

\begin{lmm}\label{lmm:sdrank}
 If $E$ is a smooth subspace field, then the function $p\mapsto \dim E_p$ is lower semi-continuous. 
\end{lmm}

\section{Some sufficient conditions for affine rigidity}

We are going to study $\ir S(\vm_{H\slt M})$ and the smooth subspace field $\ir D^h$ spanned by it.

\begin{lmm}
 $\ir D^h\subset\ker(dF)$.
\end{lmm}

\begin{proof}
 We know that $H\slt M\subset\ker dF$. So it suffices to show that if $\xi$ and $\eta$ are vector fields on $\slt M$ satisfying $\eta F = \xi F =0$, then we also have $(\operatorname{Fl}^\xi_t{}_\#\eta) F=0$.
 
 Let $c$ be an integral curve of $\xi$. Then $F$ is constant along $c$:
 \[
  (F\circ c)'(t) = \dot c(t) F = \xi(c(t))F=0.
 \]
 This implies that $F\circ \operatorname{Fl}_t^\xi = F$, wherever both sides are defined. Then
 \[
  (\operatorname{Fl}^\xi_t{}_\#\eta) F = (\eta (F\circ \operatorname{Fl}^\xi_t))\circ \operatorname{Fl}^\xi_{-t}=(\eta F)\circ \operatorname{Fl}^\xi_{-t}=0.\qedhere
 \]
\end{proof}

\begin{crl}\label{crl:imDh}
 The Finsler function $F$ is constant on the connected integral manifolds of $\ir D^h$.
\end{crl}


%
%
%
%


\begin{prp}\label{prp:rankrigid}
 If $\ir D^h$ has dimension $2n-1$ over a dense subset of $\slt M$, then $(M,F)$ is affinely rigid.
\end{prp}

\begin{proof}
Let $\bar F$ be a Finsler function for $M$ which has the same canonical spray as $F$. Fix a point $v\in\slt M$ such that $\ir D^h_v$ has dimension $2n-1$. Without loss of generality, we may assume that $F(v)=1$, and hence $v\in U(TM):=F\inv\{1\}$. By Lemma~\ref{lmm:sdrank}, $\ir D^h$ has dimension $2n-1$ on an open neighbourhood of $v$. Also, $v$ is contained in an integral manifold $N$ of $\ir D^h$, which has dimension $2n-1$. Since $F$ is constant on $N$ by Corollary~\ref{crl:imDh}, we can assume that $N$ is an open submanifold of $U(TM)$. However, $\bar F$ is also constant on $N$, thus $d(\bar F\lsz U(TM))_v=d(\bar F\lsz N)_v = 0$. Such points $v$ in $U(TM)$ form a dense set, therefore $d(\bar F\lsz U(TM))=0$. This, and the homogeneity of $F$ and $\bar F$ implies that $d(F/\bar F)=0$.
%
%
%
%
%
%
%
\end{proof}

\begin{prp}
If $U(TM):=F\inv(\{1\})$ contains countably many maximal integral manifolds of $\ir D^h$, then $(M,F)$ is affinely rigid.
\end{prp}

\begin{proof}
 Let $\bar F$ be a Finsler function for $M$ which has the same canonical spray as $F$. By Corollary~\ref{crl:imDh}, $\bar F$ is constant each integral manifold of $\ir D^h$ contained in $U(TM)$, thus $\bar F$ can have at most countably many different values on $U(TM)$. However, $\bar F$ is continuous, so this is possible only if $\bar F$ is constant on each component of $U(TM)$.
\end{proof}

\section{problems}
The following converse of Proposition~\ref{prp:rankrigid} is quite tempting:
\begin{quote}
{\it If $D^h$ has dimension less than $2n-1$ on an open subset of $\slt M$, then $(M,F)$ is not affinely rigid.}
\end{quote}

If $D^h$ has non-maximal dimension on an open subset, it can have (uncountably) many integral manifolds, which forces less rigidity on the Finsler functions that metrize the canonical spray. However, even if a smooth subspace field has non-maximal dimension on an open subset, it can still uniquely determine the functions that are constant on the integral manifolds. For example, there is a smooth subspace field $E$ on $\RR^2$ (endowed with the canonical coordinate system $(x,y)$) whose maximal integral manifolds are
\begin{enumerate}[(a)]
 \item the half-planes $y<0$ and $y>1$;
 \item the `vertical' line segments $\{(x,y)\in\RR^2\mid 0<y<1\}$, for each $x\in\RR$;
 \item each point of the straight lines given by $y=0$ or $y=1$.
\end{enumerate}
It is easy to see that the only continuous functions that are constant on each of these integral manifolds are the constant functions. Although, whether similar configuration can occur or not in the case of $\ir D^h$ is unknown.
 
For the sake of completeness, we show that there is indeed such a smooth subspace field $E$ on $\RR^2$. Let $\varphi,\psi\colon\RR\to\RR$ be smooth, nonnegative functions such that $\varphi(0)=\varphi(1) = 0$, and it is positive everywhere else, and $\psi$ vanishes on $[0,1]$, and it is positive everywhere else. Consider $\RR^2$ with its canonical coordinate system $(x,y)$, and consider the smooth subspace field $E$ spanned by the vector fields
\[
 X = (\psi\circ x)\frac{\partial}{\partial x},\qquad Y = (\varphi\circ y)\frac{\partial}{\partial y}.
\]
Then $E$ has dimension $2$ if $y>1$ or $y<0$, it has dimension $1$ if $0<y<1$, and has dimension $0$ if $y=0$ or $y=1$. Its integral manifolds are indeed the ones given above.

\section{Application}

We show how the results above can be used to characterize Finsler manifolds that admit no proper affinities. The following result is a direct analogue of \cite[p.~242, Lemma~1]{MR1393940}. 

\begin{lmm}\label{lmm:affhom}
  An affinity of a connected affinely rigid Finsler manifold is a homothety.
\end{lmm}

\begin{proof}
 Let $(M,F)$ be such a Finsler manifold, $S$ its canonical spray. Let $\varphi\colon M\to M$ be an affinity, and consider the Finsler function $\tilde F:=F\circ\varphi_*$. Obviously, $\varphi$ is an isometry from $(M,F)$ to $(M,\tilde F)$, hence the canonical spray of $(M,\tilde F)$ is the push-forward $\varphi_*{}_\#S:=\varphi_{**}\circ S\circ\varphi_*\inv$. However, $\varphi$ is also an affine transformation of $(M,F)$, so we have $\varphi_*{}_\#S = S$. Thus $F$ and $\tilde F$ have the same canonical spray. Since $(M,F)$ is affinely rigid, $\tilde F=F\circ\varphi_*$ is a constant multiple of $F$, therefore $\varphi$ is a homothety.
\end{proof}

\begin{crl}\label{crl:affiso}
  The affinities and isometries of a connected forward complete affinely rigid Finsler manifold coincide.
\end{crl}

\begin{proof}
  From the previous lemma we know that the affinities of such a Finsler manifold are homotheties. However, the only forward complete connected Finsler manifolds admitting proper homotheties are the Minkowski vector spaces \cite{MR2732642}. But Minkowski vector spaces are clearly not affinely rigid, so our claim follows.
\end{proof}

To summarize, we obtain:
\begin{thm}
 Let $(M,F)$ be a connected Finsler manifold satisfying any of the following conditions:
 \begin{enumerate}[(1)]
  \item $\operatorname{Hol}_p$ acts transitively on the unit sphere $U(T_pM):=\{v\in T_pM, F(v)=1\}$;
  \item $\ir D^h$ has dimension $2n-1$ over a dense subset of $\slt M$;
  \item $U(TM):=F\inv(\{1\})$ contains countably many maximal integral manifolds of $\ir D^h$.
 \end{enumerate}
 Then any affine transformation of $(M,F)$ is a homothety. If $(M,F)$ is also forward complete, then any affine transformation of it is an isometry. 
\end{thm}

\begin{rmrk}
 Some special cases of these results have been appeared in the literature. J.~Szenthe in \cite{MR1406352} considered the subspace field spanned by the vector fields $\vv[\xi,\eta]$ \footnote{$\vv$ is the projection with $\operatorname{ker}\vv = H\slt M$, $\operatorname{im}\vv = V\slt M$} where $\xi,\eta\in\vm_{H\slt M}$. He proved that if this subspace field has constant dimension $n-1$, then any affine transformation is a homothety. This is a special case of our result, because $\ir S(\vm_{H\slt M})$ is closed under Lie brackets by \cite[3.27 Lemma]{Michor}, and the dimension of $D^h$ is $2n-1$ if and only if the dimension of $\vv D^h$ is $n-1$. Similarly, in \cite{elgendi2016freedom} the authors considered the subspace field spanned by all the successive Lie brackets of the vector fields in $\vm_{H\slt M}$. They connected the codimension of this subspace field to the number of functionally independent Finsler functions that have the same canonical spray as $(M,F)$. As a special case they obtained that if the codimension is $1$, then the canonical spray is uniquely metrizable. Our Proposition~\ref{prp:rankrigid} is a direct generalization of this, because we consider a larger subspace field, thus it has a better chance to have the maximal dimension $2n-1$ (almost) everywhere.
\end{rmrk}

.

\end{document}